\documentclass[12pt,leqno]{article}
\usepackage{amssymb,amsmath,amsfonts,amsbsy, amsthm, xspace, latexsym, amscd, graphicx, fancyhdr,verbatim,enumitem}
\usepackage[colorlinks=true, pdfstartview=FitV, linkcolor=blue, citecolor=blue, urlcolor=blue]{hyperref}
\usepackage[normalem]{ulem}
\usepackage{color}

\theoremstyle{definition}
\newtheorem{axiom}{Axiom}
\newtheorem{myth}{Myth}
\newtheorem{fact}{Fact}
\swapnumbers
\newtheorem{definition}{Definition}[section]
\newtheorem{remark}[definition]{Remark}
\newtheorem{example}[definition]{Example}
\newtheorem{examples}[definition]{Examples}

\theoremstyle{plain}
\newtheorem{theorem}[definition]{Theorem}

\newtheorem{lemma}[definition]{Lemma}
\newtheorem{corollary}[definition]{Corollary}

\newcommand{\bra}{\ensuremath{\big\langle}}
\newcommand{\ket}{\ensuremath{\big\rangle}}
\newcommand{\supp}{\ensuremath{{\rm{supp}}}}
\newcommand{\card}{\ensuremath{{\rm{card}}}}

\newcommand{\st}{\ensuremath{{{\rm{st}}}}}

\newcommand{\rest}{\!\!\ensuremath{\upharpoonright}\!}

\newcommand{\ub}{\mathcal{UB}}
\newcommand{\cof}{\textrm{cof}}

\newlist{thm-enum}{enumerate}{2}
\newlist{pf-enum}{enumerate}{2}
\newlist{def-enum}{enumerate}{2}
\newlist{ex-enum}{enumerate}{2}
\newlist{rmk-enum}{enumerate}{2}
\setlist[thm-enum,1]{label=\textbf{\em (\roman*)},leftmargin=*,labelindent=.1\parindent}
\setlist[thm-enum,2]{label=\textbf{\em (\alph*)},leftmargin=*,labelindent=.1\parindent}
\setlist[pf-enum,1]{label=(\roman*),leftmargin=*,labelindent=.1\parindent}
\setlist[pf-enum,2]{label=(\alph*),leftmargin=*,labelindent=.1\parindent}
\setlist[def-enum,1]{label=\textbf{(\arabic*)},leftmargin=*,labelindent=.15\parindent}
\setlist[def-enum,2]{label=\textbf{(\alph*)},leftmargin=*,labelindent=.1\parindent}
\setlist[ex-enum,1]{label=\textbf{\arabic*.},leftmargin=*,labelindent=.15\parindent}
\setlist[rmk-enum,1]{label=\textbf{\arabic*.},leftmargin=*,labelindent=.15\parindent}
\begin{document}
%______________________
\title{Completeness of the Leibniz Field and Rigorousness of Infinitesimal Calculus}
%_______________________________________
\author{James F. Hall \& Todor D. Todorov\\ 
                        Mathematics Department\\                
                        California Polytechnic State University\\
                        San Luis Obispo, California 93407, USA\\																				(james@slohall.com \& ttodorov@calpoly.edu)}
\date{}
\maketitle
%__________________
\begin{abstract} We present a characterization of the completeness of the field of real numbers in the form of a \emph{collection of ten equivalent statements} borrowed from algebra, real analysis, general topology and non-standard analysis. We also discuss the completeness of non-Archimedean fields and present several examples of such fields. As an application we exploit one of our results to argue that the Leibniz infinitesimal calculus in the $18^\textrm{th}$ century was already a rigorous branch of mathematics -- at least much more rigorous than most contemporary mathematicians prefer to believe. By advocating our particular historical point of view, we hope to provoke a discussion on the importance of mathematical rigor in mathematics and science in general. We believe that our article will be of interest for those readers who teach courses on abstract algebra, real analysis, general topology, logic and the history of mathematics.
\end{abstract}

\emph{Mathematics Subject Classification:} Primary: 03H05; Secondary: 01A50, 03C20, 12J10, 12J15, 12J25, 26A06.

\emph{Key words:} Ordered field, complete field, infinitesimals, infinitesimal calculus, non-standard analysis, valuation field, power series, Hahn field, transfer principle.
%_____________________
\section{Introduction} 

	In Section~\ref{S: Preliminaries: Ordered Rings and Fields} we recall the basic definitions and results about totally ordered fields -- not necessarily Archimedean. The content of this section, although algebraic and elementary in nature, is very rarely part of standard mathematical education. In Section~\ref{S: Completeness of an Ordered Field} we present a long list of definitions of different forms of completeness of an ordered field.  These definitions are not new, but they are usually spread throughout the literature of various branches of mathematics and presented at different levels of accessibility. In Section~\ref{S: Completeness of an Archimedean Field} we present a characterization of the completeness of an Archimedean field -- that is to say, a characterization of the completeness of the reals. This characterization is in the form of a \emph{collection of ten equivalent statements} borrowed from algebra, real analysis, general topology and non-standard analysis (Theorem~\ref{T: Completeness of an Archimedean Field}). Some parts of this \emph{collection} are well-known or can be found in the literature. We, believe however, that this is the first time that the whole collection has appeared together; we also entertain the possibility that this collection is comprehensive. In Section~\ref{S: Completeness of a Non-Archimedean Field} we establish some basic results about the completeness of non-Archimedean fields which cannot be found in a  typical textbook on algebra or analysis. In Section~\ref{S: Examples of Non-Archimedean Fields} we present numerous examples of non-Archimedean fields and discuss their completeness. The main purpose of this section is to emphasize the essential difference in the completeness of Archimedean and non-Archimedean fields. 

   In Section~\ref{S: The Purge of Infinitesimals from Mathematics} we offer a short survey of the history of infinitesimal calculus, written in a \emph{polemic-like style}. One of the characterizations of the completeness of an Archimedean field presented in Theorem~\ref{T: Completeness of an Archimedean Field} is formulated in terms of infinitesimals, and thus has a strong analog in the $18^\textrm{th}$ century infinitesimal calculus. We exploit this fact, along with some older results due to J. Keisler~\cite{jKeislerF}, in order to re-evaluate ``with fresh eyes'' the rigorousness of the infinitesimal calculus. In Section~\ref{S: How Rigorous Was the Leibniz-Euler Calculus} we present a new, and perhaps surprising for many, answer to the question ``How rigorous was the infinitesimal calculus in the $18^\textrm{th}$ century?'' arguing that the \emph{Leibniz-Euler infinitesimal calculus} was, in fact, much more rigorous than most todays mathematicians prefer to believe. It seems perhaps, a little strange that it took more that 200 years to realize how the $18^\textrm{th}$ century mathematicians prefer  to phrase the \emph{completeness} of the reals. But better late (in this case, very late), then never. 

	The article establishes a connection between different topics from algebra, analysis, general topology, foundations and history of mathematics which rarely appear together in the literature. Section~\ref{S: Examples of Non-Archimedean Fields} alone -- with the help perhaps, of Section~\ref{S: Preliminaries: Ordered Rings and Fields} -- might be described as ``the shortest introduction to non-standard analysis ever written'' and for some readers might be an ``eye opener.'' We believe that this article will be of interest for all who teach courses on abstract algebra, real analysis and general topology. Some parts of the article might be accessible even for advanced students under a teacher's supervision and suitable for senior projects on real analysis, general topology or history of mathematics. We hope that the historical point of view which we advocate here might stir dormant mathematical passions resulting in a fruitful discussion on the importance of mathematical rigor to mathematics and science in general. 
%________________

\section{Preliminaries: Ordered Rings and Fields}\label{S: Preliminaries: Ordered Rings and Fields}
In this section we recall the main definitions and properties of totally ordered rings and fields (or simply, \emph{ordered rings and fields} for short), which are not necessarily Archimedean. We shortly discuss the properties of  infinitesimal, finite and infinitely large elements of a totally ordered field. For more details, and for the missing proofs, we refer the reader to (Lang~\cite{lang}, Chapter XI),  (van der Waerden~\cite{VanDerWaerden}, Chapter 11) and (Ribenboim~\cite{riben}).

\begin{definition}[Orderable Ring]\label{D: Orderable Ring} Let $\mathbb{K}$ be a ring (field). Then:
\begin{ex-enum}
  \item $\mathbb K$ is called \emph{orderable} if there exists a non-empty set $\mathbb{K}_+ \subset\mathbb{K}$ such that: (a) $0\not\in\mathbb{K}_+$; (b) $\mathbb{K}_+$ is closed under the addition and multiplication in $\mathbb{K}$; (c) For every $x\in\mathbb{K}$ exactly one of the following holds: $x=0,\,  x\in\mathbb{K}_+$ or $-x\in\mathbb{K}_+$.

\item $\mathbb K$ is \emph{formally real} if, for every $n\in\mathbb N$, the equation $\sum_{k=0}^n x_k^2 = 0$  in $\mathbb{K}^n$ admits only the trivial solution $x_1=\dots=x_n=0$.
\end{ex-enum}
\end{definition}

\begin{theorem}\label{L: orderable formally real}
  A field $\mathbb K$ is orderable if and only if $\mathbb K$ is formally real.
\end{theorem}
\begin{proof} We refer the reader to (van der Waerden~\cite{VanDerWaerden}, Chapter 11).
\end{proof}

The fields, $\mathbb Q$ and $\mathbb R$ are certainly orderable. In contrast, the field of the complex numbers $\mathbb{C}$ is non-orderable, because the equation $x^2+y^2=0$ has a non-trivial solution, $x=i, y=1$, in $\mathbb{C}$. The finite fields $\mathbb Z_p$ and the fields of the real p-adic numbers $\mathbb Q_p$ are non-orderable for similar reasons (see Ribenboim~\cite{riben} p.144-145).

\begin{definition}[Totally Ordered Rings]\label{D: Totally Ordered Rings} Let $\mathbb K$ be an orderable ring (field) and let $\mathbb{K}_+\subset\mathbb{K}$ be a set that satisfies the properties given above. Then:
   \begin{enumerate}
   \item  We define the relation $<_{\mathbb{K}_+}$ on $\mathbb{K}$ by $x<_{\mathbb{K}_+}y$ if $y-x\in\mathbb{K}_+$. We shall often write simply $<$ instead of $<_{\mathbb{K}_+}$ if the choice of $\mathbb{K}_+$ is clear from the context. Then $(\mathbb{K}, +, \cdot,  <)$, denoted for short by $\mathbb{K}$, is called a \emph{totally ordered ring (field)} or simply a \emph{ordered ring (field)} for short. If $x\in\mathbb{K}$, we define the \emph{absolute value} of $x$ by $|x|=:\max(-x,x)$.

   \item  If $A\subset\mathbb K$, we define the \emph{set of upper bounds} of $A$ by $\ub(A)=:\{x\in\mathbb K : (\forall a\in A)(a\le x)\}$. We denote by $\sup_\mathbb{K}(A)$ or, simply by $\sup(A)$, the \emph{least upper bound} of $A$ (if it exists).

   \item The \emph{cofinality} ${\rm cof}(\mathbb K)$ of $\mathbb K$ is the cardinality of the smallest unbounded subset of $\mathbb K$.

   \item Let $I$ be a directed set (Kelley~\cite{jKelley}, Chapter 2)). A net $f: I\to\mathbb K$ is called  \emph{fundamental or Cauchy} if for every $\varepsilon\in\mathbb K_+$ there exists $h\in I$ such that $|f(i)-f(j)|<\varepsilon$ for all $i, j\in I$ such that $i, j\succcurlyeq h$.

   \item Let $a, b\in\mathbb{K}$ and $a\leq b$. We let $[a, b]=:\{x\in\mathbb{K}: a\leq x\leq b\}$ and  $(a, b)=:\{x\in\mathbb{K}: a< x< b\}$.  A totally ordered ring (field) $\mathbb{K}$  will be always supplied with the order topology -- with the open intervals as basic open sets.

   \item A totally ordered ring (field) $\mathbb{K}$ is called \emph{Archimedean} if for every $x\in\mathbb{K}$, there exists $n\in\mathbb{N}$ such that $|x|\leq n$. If $\mathbb K$ is Archimedean, we also may refer to $\mathbb K(i)$ as \emph{Archimedean}. 
   \end{enumerate}
\end{definition}

\begin{theorem}[Rationals and Irrationals]\label{T: Rationals and Irrationals} 
   Let $\mathbb{K}$ be a totally ordered field. Then:
   \begin{thm-enum}

   \item $\mathbb{K}$ contains a copy of the field of the rational numbers $\mathbb{Q}$ under the order field embedding  $\sigma:\mathbb Q\to\mathbb K$ defined by: $\sigma(0)=:0$, $\sigma(n)=:n\cdot1$, $\sigma(-n)=:-\sigma(n)$ and $\sigma(\frac{m}{k})=:\frac{\sigma(m)}{\sigma(k)}$ for $n\in\mathbb N$ and $m,k\in\mathbb Z$. We shall simply write $\mathbb{Q}\subseteq\mathbb{K}$ for short.  
   \item If $\mathbb{K}\setminus\mathbb{Q}$ is non-empty, then $\mathbb{K}\setminus\mathbb{Q}$ is dense in $\mathbb{K}$ in the sense that for every $a, b\in\mathbb{K}$, such that $a<b$, there exists $x\in\mathbb{K}\setminus\mathbb{Q}$ such that $a<x<b$.
   \item If $\mathbb{K}$ is Archimedean, then $\mathbb{Q}$ is also dense in $\mathbb{K}$ in the sense that for every $a, b\in\mathbb{K}$ such that $a<b$ there exists $q\in\mathbb{Q}$ such that $a<q<b$.
   \end{thm-enum}
\end{theorem}

	 The embedding $\mathbb{Q}\subseteq\mathbb{K}$ is important for the next definition. 

\begin{definition}[Infinitesimals, etc.]\label{D: Infinitesimals, etc.}
  Let $\mathbb{K}$ be a totally ordered field. We define:
   \begin{align*}
   \mathcal{I}(\mathbb{K})&=:\{x\in\mathbb{K} : |x|< 1/n \text{ for all } n\in\mathbb N\},\\
      \mathcal{F}(\mathbb{K})&=:\{x \in\mathbb{K} : |x|\le n  \text{ for some } n\in\mathbb N\},\\
      \mathcal{L}(\mathbb{K})&=:\{x \in\mathbb{K} : n<|x|  \text{ for all } n\in\mathbb N\}.
   \end{align*}

  The elements in $\mathcal I(\mathbb K), \mathcal F(\mathbb K), \textrm{ and } \mathcal L(\mathbb K)$ are referred to as \emph{infinitesimal (infinitely small), finite and infinitely large}, respectively. We sometimes write $x\approx0$ if $x\in\mathcal I(\mathbb K)$ and $x\approx y$ if $x-y\approx 0$, in which case we say that $x$ is \emph{infinitesimally close} to $y$. If $S\subseteq\mathbb K$, we define the \emph{monad} of $S$ in $\mathbb K$ by 
   \[
   \mu(S)=\{s+dx : s\in S, dx\in\mathcal I(\mathbb K)\}.
   \]
\end{definition}

	The next result follows directly from the above definition.

\begin{lemma} Let $\mathbb K$ be a totally ordered ring. Then: (a) $\mathcal I(\mathbb K)\subset \mathcal F(\mathbb K)$; (b) $\mathbb K=\mathcal F(\mathbb K)\cup\mathcal L(\mathbb K)$;
(c) $\mathcal F(\mathbb K)\cap\mathcal L(\mathbb K)=\varnothing$. If $\mathbb{K}$ is a field, then: (d) $x\in\mathcal I(\mathbb K)$ if and only if $\frac{1}{x}\in\mathcal L(\mathbb K)$ for every non-zero $x\in \mathbb K$.
\end{lemma}

\begin{theorem}\label{T: Maximal Ideal}
  Let $\mathbb{K}$ be a totally ordered field. Then $\mathcal{F}(\mathbb{K})$ is an Archimedean ring and $\mathcal{I}(\mathbb{K})$ is a maximal ideal of $\mathcal{F}(\mathbb{K})$. Moreover, $\mathcal{I}(\mathbb{K})$ is a \emph{convex ideal} in the sense that $a\in\mathcal{F}(\mathbb{K})$ and $|a|\le b\in\mathcal{I}(\mathbb{K})$ implies $a\in\mathcal{I}(\mathbb{K})$. Consequently $\mathcal{F}(\mathbb{K})/\mathcal{I}(\mathbb{K})$ is a totally ordered Archimedean field.
\end{theorem}     

Here is a characterization of the Archimedean property ``in terms of infinitesimals.''

\begin{theorem}[Archimedean Property]\label{T: Archimedean Property}
  Let $\mathbb{K}$ be a totally ordered ring. Then the following are equivalent: (i) $\mathbb{K}$ is Archimedean. (ii) $\mathcal{F}(\mathbb{K})=\mathbb{K}$. (iii) $\mathcal L(\mathbb K)=\varnothing$. If $\mathbb{K}$ is a field, then each of the above is also equivalent to $\mathcal I(\mathbb K)=\{0\}$.  
\end{theorem}

  Notice that Archimedean rings (which are not fields) might have non-zero infinitesimals. Indeed,  if $\mathbb K$ is a non-Archimedean field, then $\mathcal F(\mathbb K)$ is always an Archimedean ring, but it has non-zero infinitesimals (see Example~\ref{Ex: Field of Rational Functions} below).

\begin{definition}[Ordered Valuation Fields]\label{D: Ordered Valuation Fields} Let $\mathbb K$ be a totally ordered field. Then:
 \begin{rmk-enum}
\item The mapping $v:\mathbb K\to \mathbb R\cup\{\infty\}$ is called a \emph{non-Archimedean Krull valuation} on $\mathbb K$ if, for every $x,y\in\mathbb K$ the properties:
\begin{description}
    \item[(a)] $v(x)=\infty$ if and only if $x=0$,
    \item[(b)] $v(xy)=v(x)+v(y)$ (\emph{Logarithmic property}),
    \item[(c)] $v(x + y)\ge\min\{v(x), v(y)\}$ (\emph{Non-Archimedean property}),
    \item[(d)] $|x| < |y|$ implies $v(x) \ge v(y)$ (\emph{Convexity property}),
  \end{description}
hold.
  The structure $(\mathbb K,v)$, denoted as $\mathbb K$ for short, is called an \emph{ordered valuation field}.
 
	\item We define the \textbf{valuation norm} $||\cdot||_v: \mathbb K\to\mathbb R$ by the formula $||x||_v=e^{-v(x)}$ with the understanding that $e^{-\infty}=0$.  Also, the formula  $d_v(x, y)=||x-y||_v$ defines the \emph{valuation metric} formula $d_v: \mathbb K\times\mathbb K\to\mathbb R$. We denote by $(\mathbb{K}, d_v)$ the \emph{associated metric space}.

\end{rmk-enum}
\end{definition}
%______________

\begin{example}[Field of Rational Functions]\label{Ex: Field of Rational Functions} Let $\mathbb K$ be an ordered field (Archimedean or not) and $\mathbb K[t]$ denote for the ring of polynomials in one variable with coefficients in $\mathbb K$. We supply the field 
\[
\mathbb K(t)=:\left\{P(t)/Q(t) : P, Q\in\mathbb K[t]\emph{ and }Q\not\equiv0\right\},
\] 
of \emph{rational functions} with ordering by:  $f< g$ in $\mathbb{K}(t)$ if there exists $n\in\mathbb{N}$ such that $g(t)-f(t)>0$ in $\mathbb{K}$ for all $t\in (0, 1/n)$. Notice that $\mathbb{K}(t)$ is a non-Archimedean field: $t, t^2, t+t^2$, etc. are positive infinitesimals, $1+t, 2+t^2, 3+t+t^2$, etc. are finite, but non-infinitesimal, and $1/t, 1/t^2, 1/(t+t^2)$, etc. are infinitely large elements of  $\mathbb{K}(t)$. Also, $\mathbb{K}(t)$ is a valuation field with valuation group $\mathbb{Z}$ and valuation $v:\mathbb{K}(t)\to \mathbb Z\cup\{\infty\}$, defined as follows: If $P\in\mathbb{K}[t]$ is a non-zero polynomial, then $v(P)$ is the lowest power of $t$ in $P$ and if $Q$ is another non-zero polynomial, then $v(P/Q)=v(P)-v(Q)$.   
\end{example}
%_____________________

\section{Completeness of an Ordered Field}\label{S: Completeness of an Ordered Field}
  
We provide a collection of definitions of several different forms of completeness of a totally ordered field -- not necessarily Archimedean. The relations between these different forms of completeness will be discussed in the next two sections. 

\begin{definition}[Completeness of a Totally Ordered Field]\label{D: Completeness of a Totally Ordered Field}
  Let $\mathbb{K}$ be a totally ordered field.
  \begin{ex-enum}
  \item If $\kappa$ is an uncountable cardinal, then $\mathbb{K}$ is called \emph{Cantor $\kappa$-complete} if every family $\{[a_\gamma,b_\gamma]\}_{\gamma\in \Gamma}$ of fewer than $\kappa$ closed bounded intervals in $\mathbb{K}$ with the finite intersection property (F.I.P.) has a non-empty intersection, $\bigcap_{\gamma\in \Gamma} [a_\gamma,b_\gamma]\neq \varnothing$.

   \item Let $^*\mathbb K$ be a non-standard extension of $\mathbb K$. Let $\mathcal F(^*\mathbb K)$ and $\mathcal I(^*\mathbb K)$ denote the sets of finite and infinitesimal elements in $^*\mathbb K$, respectively (see Definition~\ref{D: Infinitesimals, etc.}). Then we say that $\mathbb{K}$ is \emph{Leibniz complete} if every $x\in\mathcal{F}(^*\mathbb{K})$ can be presented uniquely in the form $x=r+dx$ for some $r\in\mathbb K$ and some $dx\in\mathcal{I}(^*\mathbb{K})$.  For the concept of \emph{non-standard extension of a field} we refer the reader to many of the texts on non-standard analysis, e.g. Davis~\cite{davis} or Lindstr\o m~\cite{lindstrom}. A very short definition of $^*\mathbb K$ appears also in Section~\ref{S: Examples of Non-Archimedean Fields}, Example~5, of this article. 

   \item $\mathbb K$ is \emph{Heine-Borel complete} if a subset $A\subseteq \mathbb K$ is compact if and only if $A$ is closed and bounded.

   \item We say that $\mathbb{K}$ is \emph{monotone complete} if every bounded strictly increasing sequence is convergent.

   \item We say that $\mathbb{K}$ is \emph{Cantor complete} if every nested sequence of bounded closed intervals in $\mathbb{K}$ has a non-empty intersection (that is to say that $\mathbb K$ is Cantor $\aleph_1$-complete, where $\aleph_1$ is the successor of $\aleph_0=\card(\mathbb{N})$).
 
   \item We say that $\mathbb{K}$ is \emph{Weierstrass complete} if every bounded sequence has a convergent subsequence.

   \item We say that $\mathbb{K}$ is \emph{Bolzano complete} if every bounded infinite set has a cluster point.

   \item $\mathbb{K}$ is \emph{Cauchy complete} if every fundamental $I$-net in $\mathbb K$ is convergent, where $I$ is an index set with $\card(I)={\rm cof}(\mathbb K)$. We say that $\mathbb{K}$ is simply \emph{sequentially complete} if every fundamental (Cauchy) sequence in $\mathbb K$ converges (regardless of whether or not $\cof(\mathbb K)=\aleph_0$; see Definition~\ref{D: Totally Ordered Rings}).

   \item $\mathbb{K}$ is \emph{Dedekind complete} (or \emph{order complete}) if every non-empty subset of $\mathbb{K}$ that is bounded from above has a supremum.
 
   \item Let $\mathbb{K}$ be Archimedean. Then $\mathbb{K}$ is \emph{Hilbert complete} if $\mathbb{K}$ is a maximal Archimedean field in the sense that $\mathbb K$ has no proper totally ordered Archimedean field extension.

   \item  If $\kappa$ is an infinite cardinal, $\mathbb K$ is called \emph{algebraically $\kappa$-saturated} if every family $\{(a_\gamma,b_\gamma)\}_{\gamma\in \Gamma}$ of fewer than $\kappa$ open intervals in $\mathbb K$ with the F.I.P. has a non-empty intersection, $\bigcap_{\gamma\in \Gamma} (a_\gamma,b_\gamma)\neq \emptyset$. If $\mathbb{K}$ is algebraically $\aleph_1$-saturated -- i.e. every nested sequence of open intervals has a non-empty intersection -- then we simply say that $\mathbb{K}$ is \emph{algebraically saturated}.

\item A metric space is called\, \emph{spherically complete} if every
nested sequence of closed balls has
nonempty intersection. In particular, an ordered valuation field $(\mathbb K, v)$ is \emph{spherically complete} if the associated metric space $(\mathbb K, d_v)$ is spherically complete (Definition~\ref{D: Ordered Valuation Fields}).
  \end{ex-enum}
\end{definition}
%____________________
\begin{remark}[Terminology]\label{R: Terminology} Here are some remarks about the above terminology:
  
\begin{itemize}
  \item \emph{Leibniz completeness}, listed as number 2 in Definition~\ref{D: Completeness of a Totally Ordered Field} above, appears in the early Leibniz-Euler Infinitesimal Calculus as the statement that ``every finite number is infinitesimally close to a unique usual quantity.'' Here the ``usual quantities'' are what we now refer to as the real numbers and $\mathbb{K}$ in the definition above should be identified with the set of the reals $\mathbb{R}$. We will sometimes express the Leibniz completeness as $\mathcal{F}(^*\mathbb{K})=\mu(\mathbb{K})$ (Definition~\ref{D: Infinitesimals, etc.}) which is equivalent to $\mathcal{F}(^*\mathbb{K})/\mathcal{I}(^*\mathbb{K})=\mathbb{K}$ (Theorem~\ref{T: Maximal Ideal}).

	\item \emph{Cantor $\kappa$-completeness}, monotone completeness, Weierstrass completeness, Bolzano completeness and Heine-Borel completeness typically appear in real analysis as ``theorems'' or ``important principles'' rather than as forms of completeness; however, in non-standard analysis, Cantor $\kappa$-completeness takes a much more important role along with the concept of algebraic saturation.

	 \item \emph{Cauchy completeness}, listed as number 7 above, is equivalent to the property:                                        $\mathbb K$ does not have a proper ordered field extension $\mathbb L$ such that $\mathbb K$ is dense in $\mathbb L$. The Cauchy completeness is commonly known as \emph{sequential completeness} in the particular case of Archimedean fields (and metric spaces), where $I=\mathbb N$. It has also been used in constructions of the real numbers: Cantor's construction using fundamental (Cauchy) sequences (see Hewitt \& Stromberg~\cite{hewitt} and  O'Connor~\cite{oconnor} and also Borovik \& Katz~\cite{BorovikKatz}).

  \item \emph{Dedekind completeness}, listed as number 8 above, was introduced by Dedekind (independently from many others, see O'Connor~\cite{oconnor}) at the end of the $19^\textrm{th}$ century. From the point of view of modern mathematics, Dedekind proved the consistency of the axioms of the real numbers by constructing his field of Dedekind cuts, which is an example of a Dedekind complete totally ordered field.
  
  \item \emph{Hilbert completeness}, listed as number 9 above, was originally introduced by Hilbert in 1900 with his axiomatic definition of the real numbers (see Hilbert~\cite{hilbert} and O'Connor~\cite{oconnor}).
  \end{itemize}
\end{remark}
	To the end of this section we present some commonly known facts about the Dedekind completeness (without or with very short proofs). 
%_____________________
\begin{theorem}[Existence of Dedekind Fields]\label{T: Existence of Dedekind Fields}
  There exists a Dedekind complete field.
\end{theorem}

\begin{proof} For the classical constructions of such fields due to Dedekind and Cantor, we refer the reader to  Rudin~\cite{poma} and Hewitt \& Stromberg~\cite{hewitt}, respectively. For a more recent proof of the existence of a Dedekind complete field (involving the axiom of choice) we refer to Banaschewski~\cite{bana} and for a non-standard proof of the same result we refer to Hall \& Todorov~\cite{HallTodDedekind11}.
\end{proof}
%_________________

%The terminology used above should be familiar to most, with the possible exception of 2 and 7: Leibniz complete and Hilbert complete. The term \emph{Hilbert complete} is given to the property described in 7 above because this property was Hilbert's characterization of the completeness of the reals when he presented his axiomatization of the field $\mathbb{R}$ in 1900 (see O'Connor~\cite{oconnor}). We use the term \emph{Leibniz complete} to refer to the property described in 2 above because we have observed that, in Leibniz's Infinitesimal Calculus, the set of ``usual quantities'' (as opposed to the other quantities such as infinitesimals and infinitely large elements) that Leibniz refers to resemble a rudimentary form of a field with the property that every finite quantity is infinitesimally close to a usual quantity, which is precisely what property 2 listed above says.

\begin{theorem}[Embedding]\label{T: Embedding}
  Let $\mathbb{A}$ be an Archimedean field and $\mathbb{R}$ be a Dedekind complete field. For every  $\alpha\in\mathbb{A}$ we let $C_\alpha=:\{q\in\mathbb{Q}: q<\alpha\}$. Then for every $\alpha,\beta\in\mathbb{A}$ we have: (i) $\sup_\mathbb{R}(C_{\alpha+\beta})=\sup_\mathbb{R}(C_\alpha)+\sup_\mathbb{R}(C_\beta)$.; (ii) $\sup_\mathbb{R}(C_{\alpha\beta})=\sup_\mathbb{R}(C_\alpha)\sup_\mathbb{R}(C_\beta)$; (iii) $\alpha\leq\beta$ implies $C_\alpha\subseteq C_\beta$. Consequently, the mapping $\sigma:\mathbb{A}\to\mathbb{R}$, given by $\sigma(\alpha)=:\sup_\mathbb{R}(C_\alpha)$, is an order field embedding of $\mathbb{A}$ into $\mathbb{R}$.  
\end{theorem}
	
\begin{corollary}\label{C: ded order iso}
  All Dedekind complete fields are mutually order-isomorphic and they have the same cardinality, which is usually denoted by $\mathfrak c$. Consequently,  every Archimedean field has cardinality at most $\mathfrak c$.
\end{corollary}

\begin{theorem}\label{T: Dedekind} Every Dedekind complete totally ordered field is Archimedean.
\end{theorem}

\begin{proof}
  Let $\mathbb{R}$ be such a field and suppose, to the contrary, that $\mathbb{R}$ is non-Archimedean. Then  $\mathcal L(\mathbb{R})\not=\varnothing$ by Theorem~\ref{T: Archimedean Property}. Thus $\mathbb{N}\subset\mathbb{R}$ is bounded from above by $|\lambda|$ for any $\lambda\in\mathcal L(\mathbb R)$ so that $\alpha=\sup_\mathbb R(\mathbb N)\in\mathbb{K}$ exists. Then there exists $n\in\mathbb{N}$ such that $\alpha-1< n$ implying $\alpha< n+1$, a contradiction.
\end{proof}
%_________________
\section{Completeness of an Archimedean Field}\label{S: Completeness of an Archimedean Field}

	We show that in the particular case of an Archimedean field, the different forms of completeness (1)-(10) in Definition~\ref{D: Completeness of a Totally Ordered Field} are equivalent. In the case of a non-Archimedean field, the equivalence of these different forms of completeness fails to hold -- we shall discuss this in the next section.

\begin{theorem}[Completeness of an Archimedean Field]\label{T: Completeness of an Archimedean Field} Let $\mathbb{K}$ be a totally ordered Archimedean field. Then the following are equivalent. %forms of completeness listed in Definition~\ref{D: completeness} are equivalent.
  \begin{thm-enum}
    \item $\mathbb{K}$ is Cantor $\kappa$-complete for any infinite cardinal $\kappa$.
     
    \item $\mathbb{K}$ is Leibniz complete.

	\item $\mathbb K$ is Heine-Borel complete.
                     
    \item $\mathbb{K}$ is monotone complete.
                     
    \item $\mathbb{K}$ is Cantor complete (i.e. Cantor $\aleph_1$-complete, not for all cardinals).
                     
    \item $\mathbb{K}$ is Weierstrass complete.
      
    \item $\mathbb{K}$ is Bolzano complete.
                     
    \item $\mathbb{K}$ is Cauchy complete.
                     
    \item $\mathbb{K}$ is Dedekind complete.

    \item $\mathbb{K}$ is Hilbert complete.
  \end{thm-enum}
\end{theorem}
 
\begin{proof}
  \mbox{} 
  \begin{description}
  \item[$(i)\Rightarrow(ii)$:] Let $\kappa$ be the successor of $\card(\mathbb{K})$. Let $x\in\mathcal F(^*\mathbb K)$ and $S=:\{[a,b]:a,b\in\mathbb K\emph{ and } a\le x\le b\textrm{ in }{^*\mathbb K}\}$. Clearly $S$ satisfies the finite intersection property and $\card(S)=\card(\mathbb{K}\times\mathbb{K})=\card(\mathbb{K})<\kappa$; thus, by assumption, there exists $r\in \bigcap_{[a,b]\in S} [a,b]$. To show $x-r\in\mathcal I(^*\mathbb K)$, suppose (to the contrary) that $\frac{1}{n}<|x-r|$ for some $n\in\mathbb N$. Then either $x<r-\frac{1}{n}$ or $r+\frac{1}{n}<x$. Thus (after letting $r-\frac{1}{n}=b$ or $r+\frac{1}{n}=a$) we conclude that either $r\le r-\frac{1}{n}$, or $r+\frac{1}{n}\le r$, a contradiction. 

 \item[$(ii)\Rightarrow(iii)$:] Our assumption (ii) justifies the following definitions: We define $\st: \mathcal{F}(^*\mathbb K)\to\mathbb K$ by $\st(x)=r$ for $x=r+dx$, $dx\in\mathcal I({^*\mathbb K})$. Also, if $S\subset\mathbb K$, we let $\st[^*S]=\{\st(x): x\in{^*S}\cap\mathcal{F}(^*\mathbb K)\}$.  If $S$ is compact, then $S$ is bounded and closed since $\mathbb K$ is a Hausdorff space as an ordered field. Conversely, if $S$ is bounded and closed, it follows that $^*S\subset\mathcal{F}(^*\mathbb K)$ (Davis~\cite{davis}, p. 89) and $\st[^*S]=S$ (Davis~\cite{davis}, p. 77), respectively. Thus $^*S\subset\mu(S)$, i.e. $S$ is compact (Davis~\cite{davis}, p. 78). 
 
  \item[$(iii)\Rightarrow(iv)$:] Let $(x_n)$ be a bounded from above strictly increasing sequence in $\mathbb K$ and let $A=\{x_n\}$ denote the range of the sequence. Clearly $\overline A\setminus A$ is either empty or contains a single element which is the limit of $(a_n)$; hence it suffices to show that $\overline A\neq A$. To this end, suppose, to the contrary, that $\overline A=A$. Then we note that $A$ is compact by assumption since $(a_n)$ is bounded; however, if we define $(r_n)$ by $r_1=1/2(x_2 - x_1)$, $r_n=\min\{r_1,\ldots, r_{n-1}, 1/2(x_{n+1} - x_n)\}$, then we observe that the sequence of open intervals $(U_n)$, defined by $U_n = (x_n-r_n, x_n+r_n)$, is an open cover of $A$ that has no finite subcover. Indeed, $(U_n)$ is pairwise disjoint so that every finite subcover contains only a finite number of terms of the sequence. The latter contradicts the compactness of $A$. 
 
  \item[$(iv)\Rightarrow(v)$:] Suppose that $\{[a_i,b_i]\}_{i\in\mathbb{N}}$ satisfies the finite intersection property. Let $\Gamma_n=:\cap_{i=1}^n[a_i,b_i]$ and observe that $\Gamma_n=[\alpha_n, \beta_n]$ where $\alpha_n=:\max_{i\le n} a_i$ and $\beta_n=:\min_{i\le n}b_i$. Then $\{\alpha_n\}_{n\in\mathbb N}$ and $\{-\beta_n\}_{n\in\mathbb N}$ are bounded increasing sequences; thus $\alpha=:\lim_{n\to\infty}\alpha_n$ and $-\beta=:\lim_{n\to\infty}-\beta_n$ exist by assumption. If $\beta<\alpha$, then for some $n$ we would have $\beta_n<\alpha_n$, a contradiction; hence, $\alpha\le\beta$. Therefore $\cap_{i=1}^{\infty}[a_i,b_i]=[\alpha,\beta]\neq\varnothing$.

  \item[$(v)\Rightarrow(vi)$:] This is the familiar \emph{Bolzano-Weierstrass Theorem} (Bartle \& Sherbert~\cite{bartle}, p. 79). 

 \item[$(vi)\Rightarrow(vii)$:] Let $A\subset \mathbb{K}$ be a bounded infinite set. By the Axiom of Choice, $A$ has a denumerable subset -- that is, there exists an injection $\{x_n\}:\mathbb{N}\to A$. As $A$ is bounded, $\{x_n\}$ has a subsequence $\{x_{n_k}\}$ that converges to a point $x\in\mathbb{K}$ by assumption. Then $x$ must be a cluster point of $A$ because the sequence $\{x_{n_k}\}$ is injective, and thus not eventually constant.

  \item[$(vii)\Rightarrow(viii)$:] For the index set we can assume that $I=\mathbb N$ since cofinality of any Archimedean set is $\aleph_0=\card(\mathbb N)$. Let $\{x_n\}$ be a Cauchy sequence in $\mathbb{K}$. Then $\mathrm{range}(\{x_n\})$ is a bounded set. If $\mathrm{range}(\{x_n\})$ is finite, then $\{x_n\}$ is eventually constant (and thus convergent). Otherwise, $\mathrm{range}(\{x_n\})$ has a cluster point $L$ by assumption. To show that $\{x_n\}\to L$, let $\epsilon\in\mathbb{K}_+$ and $N\in\mathbb{N}$ be such that $n,m\ge N$ implies that $|x_n-x_m|<\frac{\epsilon}{2}$. Observe that the set $\{n\in\mathbb{N} : |x_n-L|<\frac{\epsilon}{2} \}$ is infinite because $L$ is a cluster point, so that $A=:\{n\in\mathbb{N} : |x_n-L|<\frac{\epsilon}{2}, n\ge N\}$ is non-empty. Let $M=:\min A$. Then, for $n\ge M$, we have $|x_n-L|\le|x_n-x_M|+|x_M-L|<\epsilon$, as required.
  
  \item[$(viii)\Rightarrow(ix)$:] This proof can be found in (Hewitt \& Stromberg~\cite{hewitt}, p. 44).

  \item[$(ix)\Rightarrow(x)$:] Let $\mathbb{A}$ be a totally ordered Archimedean field extension of $\mathbb{K}$. We have to show that $\mathbb{A}=\mathbb{K}$. Recall that $\mathbb{Q}$ is dense in $\mathbb{A}$ as it is Archimedean; hence, the set $\{q\in \mathbb{Q} : q<a\}$ is non-empty and bounded above in $\mathbb{K}$ for all $a\in\mathbb{A}$. Consequently, the mapping 
$\sigma:\mathbb{A}\to\mathbb{K}$, where $\sigma(a)=:\sup_\mathbb{K}\{q\in\mathbb{Q} : q<a\}$, is well-defined by our assumption. Note that $\sigma$ fixes $\mathbb{K}$. To show that $\mathbb{A}=\mathbb{K}$ we will show that $\sigma$ is just the identity map. Suppose (to the contrary) that $\mathbb{A}\neq\mathbb{K}$ and let $a\in\mathbb{A}\setminus\mathbb{K}$. Then $\sigma(a)\neq a$ so that either $\sigma(a)>a$ or $\sigma(a)<a$. If it is the former, then there exists $p\in\mathbb{Q}$ such that $a<p<\sigma(a)$, contradicting the fact that $\sigma(a)$ is the \emph{least} upper bound for $\{q\in\mathbb{Q} : q<a\}$ and if it is the latter then there exists $p\in\mathbb{Q}$ such that $\sigma(a)<p<a$, contradicting the fact that $\sigma(a)$ is an upper bound for $\{q\in\mathbb{Q} : q<a\}$.  
    %Suppose that $\mathbb{F}$ is a totally ordered Archimedean field extension of $\mathbb{K}$. Since $\mathbb{F}$ is Archimedean, $\mathbb{Q}$ is dense in $\mathbb{F}$ so that for every $x\in \mathbb{F}$, the set
    %\[S_x=:\{q\in \mathbb{Q} : q<x\}\]
    %is non-empty and bounded above. Define $\sigma:\mathbb{F}\to\mathbb{K}$ by $\sigma(x)=\sup S_x$. Then $\sigma$ is an embedding of $\mathbb{F}$ into $\mathbb{K}$, so that $\mathbb{F}$ must be a subfield of $\mathbb{K}$, thus $\mathbb{F}\cong\mathbb{K}$. Therefore $\mathbb{K}$ does not have a totally ordered Archimedean proper field extension.

  \item[$(x)\Rightarrow(i)$:]  Let $\mathbb{D}$ be a Dedekind complete field (such a field exists by Theorem~\ref{T: Existence of Dedekind Fields}). We can assume that $\mathbb{K}$ is an ordered subfield of $\mathbb{D}$ by Theorem~\ref{T: Embedding}. Thus we have $\mathbb{K}=\mathbb{D}$ by assumption, since $\mathbb D$ is Archimedean. Now, suppose (to the contrary) that there is an infinite cardinal $\kappa$ and a family $[a_i,b_i]_{i\in I}$ of fewer than $\kappa$ closed bounded intervals in $\mathbb K$ with the finite intersection property such that $\bigcap_{i\in I}[a_i,b_i]=\varnothing$.  Because $[a_i,b_i]$ satisfies the finite intersection property, the set $A=:\{a_i :i\in I\}$ is bounded from above and non-empty so that $c=:\sup(A)$ exists in $\mathbb{D}$. Thus $a_i\le c\le b_i$ for all $i\in I$ so that $c\in\mathbb{D}\setminus\mathbb{K}$. Thus $\mathbb{D}$ is a proper field extension of $\mathbb{K}$, a contradiction. 
\end{description}
\end{proof}

\begin{remark}
  It should be noted that the equivalence of $(ii)$ and $(ix)$ above is proved in Keisler (\cite{jKeislerF}, pp. 17-18) with somewhat different arguments. Also, the equivalence of $(ix)$ and $(x)$ is proved in Banaschewski~\cite{bana} using a different method than ours (with the help of the axiom of choice).
\end{remark}
%__________________________

\section{Completeness of a Non-Archimedean Field}\label{S: Completeness of a Non-Archimedean Field}

In this section, we discuss how removing the assumption that $\mathbb K$ is Archimedean affects our result from the previous section. In particular, several of the forms of completeness listed in Definition~\ref{D: Completeness of a Totally Ordered Field} no longer hold, and those that do are no longer equivalent. 

\begin{theorem}\label{T: completeness -> archimedean}
   Let $\mathbb K$ be an ordered field satisfying any of the following:
   \begin{thm-enum}
      \item Bolzano complete.
      \item Weierstrass complete.
      \item Monotone complete.
      \item Dedekind complete
      \item Cantor $\kappa$-complete for $\kappa>\card(\mathbb K)$.
      \item Leibniz complete (in the sense that every finite number can be decomposed \emph{uniquely} into the sum of an element of $\mathbb K$ and an infinitesimal).
   \end{thm-enum}
 Then $\mathbb K$ is Archimedean. Consequently, if $\mathbb K$ is non-Archimedean, then each of (i)-(vi) is false. 
\end{theorem}
\begin{proof}
   We will only prove the case for Leibniz completeness and leave the rest to the reader.

  Suppose, to the contrary, that $\mathbb K$ is non-Archimedean. Then there exists a $dx\in\mathcal I(\mathbb K)$ such that $dx\neq0$ by Theorem~\ref{T: Archimedean Property}. Now take $\alpha\in\mathcal F(^*\mathbb K)$ arbitrarily. By assumption there exists unique $k\in\mathbb K$ and $d\alpha\in\mathcal I(^*\mathbb K)$ such that $\alpha=k+d\alpha$. However, we know that $dx\in\mathcal I(^*\mathbb K)$ as well because $\mathbb{K}\subset{^*\mathbb{K}}$ and the ordering in $^*\mathbb{K}$ extends that of $\mathbb{K}$. Thus $(k+dx)+(d\alpha-dx)=k+d\alpha=\alpha$ where $k+dx\in\mathbb K$ and $d\alpha-dx\in\mathcal I(^*\mathbb K)$. This contradicts the uniqueness of $k$ and $d\alpha$. Therefore $\mathbb K$ is Archimedean.
\end{proof}

As before, $\kappa^+$ stands for the successor of $\kappa$, $\aleph_1=\aleph_0^+$ and $\aleph_0=\card(\mathbb N)$.
\begin{theorem}[Cardinality and Cantor Completeness]\label{T: Cardinality and Cantor Completeness}
  Let $\mathbb K$ be an ordered field. If $\mathbb K$ is non-Archimedean and Cantor $\kappa$-complete (see Definition~\ref{D: Completeness of a Totally Ordered Field}), then $\kappa\le\card(\mathbb K)$.
\end{theorem}
\begin{proof}
   This is essentially the proof of $(i)\Rightarrow(ii)$ in Theorem~\ref{T: Completeness of an Archimedean Field}.
  %Suppose, to the contrary, that $\kappa>\card(\mathbb K)$. Let $S\subset\mathbb K$ be bounded from above and $\Gamma=:\{[a,b] : a\in S, b\in\ub(S)\}$. Observe that $\Gamma$ satisfies the finite intersection property and that $\card(\Gamma)\le\card(\mathbb K)\times\card(\mathbb K)=\card(\mathbb K)$. Then there exists $\sigma\in\bigcap_{\gamma\in\Gamma} \gamma$ by assumption. Clearly $(\forall a\in S)(a\le\sigma)$, thus $\sigma\in\ub(S)$. But we also have $(\forall b\in\ub(S))(\sigma\le b)$. Therefore $\sigma=\sup(S)$. Thus it follows from Theorem~\ref{T: completeness -> archimedean} that $\mathbb{K}$ is Archimedean, a contradiction. 
\end{proof}

	In the proof of the next result we borrow some arguments from a similar (unpublished) result due to Hans Vernaeve.

\begin{theorem}[Cofinality and Saturation]\label{T: Cofinality and Saturation}
  Let $\mathbb K$ be an ordered field and $\kappa$ be an uncountable cardinal. Then the following are equivalent:
  \begin{thm-enum}
    \item $\mathbb K$ is algebraically $\kappa$-saturated.
    \item $\mathbb K$ is Cantor $\kappa$-complete and $\cof(\mathbb K)\ge\kappa$.
  \end{thm-enum}
\end{theorem}
\begin{proof}
  \mbox{}
  \begin{description}
    \item[$(i)\Rightarrow(ii)$:] Let $\mathcal C=:\{[a_{\gamma},b_{\gamma}]\}_{\gamma\in \Gamma}$ and $\mathcal O=:\{(a_{\gamma}, b_{\gamma})\}_{\gamma\in \Gamma}$ be families of fewer than $\kappa$ bounded closed and open intervals, respectively, where $\mathcal C$ has the F.I.P.. If $a_k=b_p$ for some $k,p\in\Gamma$, then  $\bigcap_{\gamma\in \Gamma}[a_{\gamma},b_{\gamma}]=\{a_k\}$ by the F.I.P. in $\mathcal C$. Otherwise, $\mathcal O$ has the F.I.P.; thus, there exists $\alpha\in\bigcap_{\gamma\in \Gamma} (a_{\gamma}, b_{\gamma})\subseteq\bigcap_{\gamma\in \Gamma}[a_{\gamma}, b_{\gamma}]$ by algebraic $\kappa$-saturation. Hence $\mathbb K$ is Cantor $\kappa$-complete. To show that the cofinality of $\mathbb K$ is greater than or equal to $\kappa$, let $A\subset\mathbb K$ be a set with $\card(A)<\kappa$. Then $\bigcap_{a\in A}(a,\infty)\neq\emptyset$ by algebraic $\kappa$-saturation.

    \item[$(ii)\Rightarrow(i)$:] Let $\{(a_{\gamma},b_{\gamma})\}_{\gamma\in \Gamma}$ be a family of fewer than $\kappa$ elements with the F.I.P.. Without loss of generality, we can assume that each interval is bounded. As $\cof(\mathbb K)\ge\kappa$, there exists $\frac{1}{\rho} \in \ub(\{ \frac{1}{b_l -a_k} : l, k\in \Gamma \})$ (that is, $\frac{1}{b_l-a_k}\le\frac{1}{\rho}$ for all $l,k\in \Gamma$) which implies that $\rho>0$ and that $\rho$ is a lower bound of $\{b_l-a_k : l, k\in \Gamma\}$. Next, we show that the family $\{[a_\gamma+\frac{\rho}{2},b_\gamma-\frac{\rho}{2}]\}_{\gamma\in\Gamma}$ satisfies the F.I.P.. Let $\gamma_1,\ldots,\gamma_n\in\Gamma$ and $\zeta=:\max_{k\le n}\{a_{\gamma_k} + \frac{\rho}{2}\}$. Then, for all $m\in\mathbb N$ such that $m\le n$, we have $a_{\gamma_m} + \frac{\rho}{2}\le \zeta \le b_{\gamma_m} - \frac{\rho}{2}$ by the definition of $\rho$; thus, $\zeta\in[a_{\gamma_m}+\frac{\rho}{2}, b_{\gamma_m}-\frac{\rho}{2}]$ for $m\le n$. By Cantor $\kappa$-completeness, there exists $\alpha\in\bigcap_{\gamma\in\Gamma} [a_{\gamma}+\frac{\rho}{2}, b_{\gamma}-\frac{\rho}{2}]\subseteq\bigcap_{\gamma\in\Gamma}(a_{\gamma},b_{\gamma})$.
  \end{description}
\end{proof}

\begin{lemma}\label{L: sat -> seq}
  Let $\mathbb K$ be an ordered field. If $\mathbb K$ is algebraically saturated, then $\mathbb K$ is sequentially complete.
\end{lemma}
\begin{proof}
   Let $\{x_n\}$ be a Cauchy sequence in $\mathbb K$. Define $\{\delta_n\}$ by $\delta_n=|x_n-x_{n+1}|$. If $\{\delta_n\}$ is not eventually constant, then there is a subsequence $\{\delta_{n_k}\}$ such that $\delta_{n_k}>0$ for all $k$; however, this yields $\bigcap_k (0,\delta_{n_k})\neq\varnothing$, which contradicts $\delta_n\to 0$. Therefore $\{\delta_n\}$ is eventually zero so that $\{x_n\}$ is eventually constant.
\end{proof}

\begin{lemma}\label{L: cantor -> sequential}
  Let $\mathbb K$ be an ordered field. If $\mathbb K$ is Cantor complete, but not algebraically saturated, then $\mathbb K$ is sequentially complete.
\end{lemma}
\begin{proof}
   By Theorem~\ref{T: Cofinality and Saturation}, we know there exists an unbounded increasing sequence $\{\frac{1}{\epsilon_n}\}$. Let $\{x_n\}$ be a Cauchy sequence in $\mathbb K$. For all $n\in\mathbb N$, we define $S_n=:[x_{m_n} - \epsilon_n, x_{m_n} + \epsilon_n]$, where $m_n\in\mathbb N$ is the minimal element such that $l,j\ge m_n$ implies $|x_l-x_j|<\epsilon_n$. Let $A\subset \mathbb N$ be finite and $\rho=:\max(A)$; then we observe that $x_{m_{\rho}}\in S_k$ for any $k\in A$ because $m_k\le m_{\rho}$; hence $\{S_n\}$ satisfies the F.I.P.. Therefore there exists $L\in\bigcap_{k=1}^{\infty}S_k$ by Cantor completeness. It then follows that $x_n\to L$ since $\{\epsilon_n\}\to0$.
\end{proof}

\begin{theorem}\label{T: non-arch completeness}
  Let $\mathbb K$ be an ordered field, then we have the following implications
 \[
\mathbb K\textrm{ is }\kappa\textrm{-saturated}\Rightarrow\mathbb K\textrm{ is Cantor }\kappa\textrm{-complete}\Rightarrow \mathbb K\textrm{ is sequentially complete}.
\]
\end{theorem}
\begin{proof}
   The first implication follows from Theorem~\ref{T: Cofinality and Saturation}. For the second we have two cases depending on whether $\mathbb K$ is algebraically saturated, which are handled by Lemmas~\ref{L: sat -> seq} and Lemma~\ref{L: cantor -> sequential}.
\end{proof}

%_____________________

\section{Examples of Non-Archimedean Fields}\label{S: Examples of Non-Archimedean Fields}

	In this section we present several examples of non-Archimedean fields which, on one hand, illustrate some of the results from the previous sections, but on the other prepare us for the discussion of the history of calculus in the next section. This section alone -- with the help perhaps, of Section~\ref{S: Preliminaries: Ordered Rings and Fields} -- might be described as ``the shortest introduction to non-standard analysis ever written'' and for some readers might be an ``eye opener.'' 

	If $X$ and $Y$ are two sets, we denote by $Y^X$ the set of all functions from $X$ to $Y$. In particular, $Y^\mathbb N$ stands for the set of all sequences in $Y$. We use the notation $\mathcal{P}(X)$ for the power set of $X$. To simplify the following discussion, we shall adopt the GCH (Generalized Continuum Hypothesis) in the form $2^{\aleph_\alpha}=\aleph_{\alpha+1}$ for all ordinals $\alpha$, where $\aleph_\alpha$ are the cardinal numbers. Also, we let $\card(\mathbb N)=\aleph_0$ and $\card(\mathbb R) =\frak{c}\, (=\aleph_1)$ and we shall use ${\frak{c}^+}\, (=\aleph_2)$ for the successor of $\frak{c}$. Those readers who do not like cardinal numbers are advised to ignore this remark. 

In what follows, $\mathbb{K}$ is a totally ordered field (Archimedean or not) with cardinality $\card(\mathbb K)=\kappa$ and cofinality ${\rm cof}(\mathbb K)$ (Definition~\ref{D: Totally Ordered Rings}). Let $\mathbb K(t)$ denotes the field of the rational functions (Example~\ref{Ex: Field of Rational Functions}). Then we have the following examples of non-Archimedean fields:
%________
 \begin{equation}\label{E: Chain 1}
\mathbb K\subset \mathbb K(t)\subset \mathbb K(t^{\mathbb Z})\subset \mathbb K\langle t^{\mathbb R}\rangle\subset \mathbb K((t^{\mathbb R})),
\end{equation}
where:
\begin{ex-enum}

\item The field of \emph{Hahn series with coefficients in $\mathbb K$ and valuation group} $\mathbb R$ is defined to be the set 
\[
\mathbb K((t^{\mathbb R}))=: \Big\{S=\sum_{r\in\mathbb R} a_rt^r : a_r\in\mathbb K\emph{ and } \supp(S) \emph{ is a well ordered set}\Big\},
\]
where $\supp(S)=\{r\in\mathbb R: a_r\not=0\}$. We supply $\mathbb K((t^{\mathbb R}))$ (denoted sometimes by $\mathbb K(\mathbb R)$) with the usual polynomial-like addition and multiplication and the \emph{canonical valuation} $\nu:\mathbb K((t^{\mathbb R}))\to\mathbb R\cup\{\infty\}$ defined by $\nu(0)=:\infty$ and $\nu(S)=:\min(\supp(S))$ for all $S\in\mathbb K((t^{\mathbb R})), S\not= 0$. As well, $\mathbb K((t^{\mathbb R}))$ has a natural ordering given by 
\[
\mathbb K((t^{\mathbb R}))_+=:\Big\{S=\sum_{r\in\mathbb R} a_rt^r : a_{\nu(S)}>0\Big\}.
\]

\item The field of Levi-Civita series is defined to be the set 
\[
\mathbb K\langle t^{\mathbb R}\rangle=:\Big\{\sum_{n=0}^{\infty}a_nt^{r_n} : a_n\in\mathbb K \emph{ and } (r_n)\in\mathbb R^\mathbb N\Big\},
\]
where the sequence $(r_n)$ is required to be strictly increasing and unbounded.

\item $\mathbb K(t^{\mathbb Z})=:\Big\{\sum_{n=m}^{\infty}a_nt^n : a_n\in\mathbb K \emph{ and }m\in\mathbb Z\Big\}$ is the field of formal \emph{Laurent series} with coefficients in $\mathbb K$.

Both $\mathbb K(t^{\mathbb Z})$ and $\mathbb K\langle t^{\mathbb R}\rangle$ are supplied with algebraic operations, ordering and valuation inherited from $\mathbb K((t^{\mathbb R}))$.

\item Since $\mathbb K(t)$ is Non-Archimedean (Example~\ref{Ex: Field of Rational Functions}), so are the fields   $\mathbb K(t^\mathbb Z)$, $\mathbb K\bra t^\mathbb R\ket$ and $\mathbb K((t^\mathbb R))$. Here is an example for a positive infinitesimal, $\sum_{n=0}^{\infty}n!\, t^{n+1/3}$, and a positive infinitely large number, $\sum_{n=-1}^{\infty} t^{n+1/2}=\frac{\sqrt{t}}{t-t^2}$, both in $\mathbb K\bra t^\mathbb R\ket$. If $\mathbb K$ is real closed, then both $\mathbb K\langle t^{\mathbb R}\rangle$ and $\mathbb K((t^{\mathbb R}))$ are real closed (Prestel~\cite{aPrestel}). The cofinality of each of the fields $\mathbb K(t^{\mathbb Z})$, $\mathbb K\langle t^{\mathbb R}\rangle$, and $\mathbb K((t^{\mathbb R}))$ is $\aleph_0=\card(\mathbb N)$, because the sequence $(1/t^n)_{n\in\mathbb N}$ is unbounded in each of them. The field $\mathbb{K}((t^{\mathbb{R}}))$ is spherically complete by (Krull~\cite{krull} and Luxemburg~\cite{wLux76}, Theorem~2.12). Consequently, the fields  $\mathbb K(t^{\mathbb Z})$, $\mathbb K\langle t^{\mathbb R}\rangle$, and $\mathbb K((t^{\mathbb R}))$ are Cauchy (and sequentially) complete. Neither of these fields is necessarily Cantor complete or saturated. If $\mathbb K$ is Archimedean, these fields are certainly not Cantor complete. The fact that the series $\mathbb{R}(t^{\mathbb Z})$ and $\mathbb{R}\langle t^{\mathbb R}\rangle$ are sequentially complete was also proved independently in (Laugwitz~\cite{Laugwitz}). The field $\mathbb{R}\langle t^{\mathbb R}\rangle$ was introduced by
Levi-Civita in \cite{LC} and later was investigated by D. Laugwitz in
\cite{Laugwitz} as a potential framework for the rigorous foundation of
infinitesimal calculus before the advent of Robinson's nonstandard
analysis. It is also an example of a real-closed valuation field that is
sequentially complete, but not spherically complete (Pestov~\cite{vPestov}, p. 67).

\item Let $^*\mathbb K=\mathbb K^{\mathbb N}/{\sim}$ be a non-standard extension of $\mathbb K$. Here $\mathbb K^\mathbb N$ stands for the ring of all sequences in $\mathbb K$ and $\sim$ is an equivalence relation on $\mathbb K^\mathbb N$ defined in terms of a (fixed) free ultrafilter $\mathcal U$ on $\mathbb N$: $(a_n)\sim (b_n)$ if $\{n\in \mathbb N: a_n=b_n\}\in\mathcal{U}$.  We denote by $\bra a_n\ket\in{^*\mathbb K}$ the equivalence class of the sequence $(a_n)\in{\mathbb K^\mathbb N}$. The \emph{non-standard extension} of a set $S\subseteq\mathbb K$ is defined by $^*S=\big\{\bra a_n\ket\in{^*\mathbb K}: \{n\in \mathbb N: a_n\in S\}\in\mathcal{U}\big\}$.  It turns out that $^*\mathbb K$ is a  algebraically saturated ($\frak{c}$-saturated) ordered field which contains a copy of $\mathbb K$ by means of the constant sequences. The latter implies $\card(^*\mathbb K)\ge \frak{c}$ by Theorem~\ref{T: Cardinality and Cantor Completeness} and ${\rm cof}(^*\mathbb K)\geq\frak{c}$ by Theorem~\ref{T: Cofinality and Saturation}. It can be proved that $^*\mathbb N$ is unbounded from above in $^*\mathbb K$, i.e. $\bigcap_{n\in{^*\mathbb N}}(n, \infty)=\varnothing$. The latter implies $\card(^*\mathbb N)\geq\frak{c}$ by the $\frak{c}$-saturation of $^*\mathbb K$ and thus ${\rm cof}(^*\mathbb K)\le\card(^*\mathbb N)$. Finally, $^*\mathbb K$ is real closed if and only if $\mathbb K$ is real closed. If $r\in\mathbb K, r\not=0$, then  $\bra 1/n\ket, \bra r+1/n\ket, \bra n\ket$ present examples for a positive infinitesimal, finite (but non-infinitesimal) and infinitely large elements in $^*\mathbb K$, respectively.  Let $X\subseteq\mathbb K$ and $f: X\to\mathbb K$ be a real function. We defined its \emph{non-standard extension} $^*f: {^*X}\to{^*\mathbb K}$ by the formula $^*f(\bra x_n\ket)=\bra f(x_n)\ket$. It is clear that $^*f\rest X=f$, hence we can sometimes skip the asterisks to simplify our notation. Similarly we define $^*f$ for functions $f: X\to\mathbb K^q$, where $X\subseteq\mathbb K^p$ (noting that $^*(\mathbb K^p)={(^*\mathbb K)^p}$). Also if $f\subset \mathbb K^p\times \mathbb K^q$ is a function, then $^*f\subset{^*\mathbb K^p}\times{^*\mathbb K^q}$ is as well. For the details and missing proofs of this and other results we refer to any of the many excellent introductions to non-standard analysis, e.g.  Lindstr\o m~\cite{lindstrom}, Davis~\cite{davis}, Capi\'{n}ski \& Cutland ~\cite{CapinskiCutland95} and Cavalcante~\cite{cavalcante}.

\begin{remark}[Free Ultrafilter] 
   Recall that $\mathcal{U}\subset\mathcal{P}(\mathbb N)$ is a \emph{free ultrafilter} on $\mathbb N$ if: (a) $\varnothing\notin\mathcal{U}$; (b) $\mathcal{U}$ is closed under finitely many intersections; (c) If $A\in\mathcal{U}$ and $B\subseteq\mathbb N$, then $A\subseteq B$ implies $B\in\mathcal{U}$; (d) $\bigcap_{A\in\mathcal{U}}\, A=\varnothing$; (e) For every $A\subseteq\mathbb N$ either $A\in\mathcal{U}$, or $\mathbb N\setminus A\in\mathcal{U}$. Recall as well that the existence of free ultrafilters follows from the axiom of choice (Zorn's lemma). For more details we refer again to (Lindstr\o m~\cite{lindstrom}).
\end{remark}

\begin{theorem}[Leibniz Transfer Principle]\label{T: Leibniz Transfer Principle} 
   Let $\mathbb K$ be an ordered Archimedean field and $p, q\in\mathbb N$. Then $S$ is the solution set of the system
   \begin{equation}\label{E: System in K}
   \begin{cases}
   f_i(x)=F_i(x), &i=1, 2, \dots, n_1,\\
   g_j(x)\not= G_j(x),&j=1, 2, \dots, n_2,\\
   h_k(x)\leq H_k(x), &k=1, 2, \dots, n_3,
   \end{cases}
   \end{equation}
   if and only if $^*S$ is the solution set of the system of equations and inequalities
   \begin{equation}\label{E: System in *K}
   \begin{cases}
   ^*f_i(x)={^*}F_i(x), &i=1, 2, \dots, n_1,\\
   ^*g_j(x)\not={^*}G_j(x), &j=1, 2, \dots, n_2,\\
   ^*h_k(x)\leq {^*}H_k(x), &k=1, 2, \dots, n_3.
   \end{cases}
   \end{equation}
Here $f_i, F_i, g_j, G_j\subset \mathbb K^p\times \mathbb K^q$ and  $h_k, H_k\subset \mathbb K^p\times \mathbb K$ are functions in $p$-variables and $n_1, n_2, n_3\in\mathbb{N}_0$ (if $n_1=0$, then $f_i(x)=F_i(x)$ will be missing in (\ref{E: System in K}) and similarly for the rest). The quantifier ``for all'' is over $d, p, q$ and all functions involved.
\end{theorem}

\item Let  $^*\mathbb R$ be the non-standard extension of $\mathbb R$ (the previous example for $\mathbb K=\mathbb R$). The elements of $^*\mathbb R$ are known as \emph{non-standard real numbers}. $^*\mathbb R$ is a real closed algebraically saturated ($\frak{c}$-saturated) field in the sense that every nested sequence of open intervals in $^*\mathbb R$ has a non-empty intersection.  Also, $\mathbb R$ is embedded as an ordered subfield of ${^*\mathbb R}$ by means of the constant sequences. We have $\card(^*\mathbb R)=\frak{c}$ (which means that $^*\mathbb R$ is fully saturated). Indeed, in addition to $\card(^*\mathbb R)\ge\frak{c}$ (see above), we have  $\card(^*\mathbb R)\le \card(\mathbb R^\mathbb N)= (2^{\aleph_0})^{\aleph_0}=2^{(\aleph_0)^2}=2^{\aleph_0}=\frak{c}$. We also have ${\rm cof}(^*\mathbb R)=\frak{c}$. Indeed, in addition to  ${\rm cof}(^*\mathbb R)\geq\frak{c}$ (see above) we have ${\rm cof}(^*\mathbb R)\le\card(^*\mathbb R)=\frak{c}$. Here are several important results of non-standard analysis:

\begin{theorem}[Leibniz Completeness Principle] $\mathbb R$ is \emph{Leibniz complete} in the sense that every finite number $x\in{^*\mathbb R}$ is infinitely close to some (necessarily unique) number $r\in\mathbb R$ (\# 2 of Definition~\ref{D: Completeness of a Totally Ordered Field}). We define the standard part mapping $\st: \mathcal{F}(^*\mathbb R)\to\mathbb R$ by $\st(x)=r$. 
\end{theorem}

\begin{theorem}[Leibniz Derivative]\label{T: Leibniz Derivative}  Let  $X\subseteq\mathbb R$ and $r$ be a non-trivial adherent (cluster) point of $X$. Let $f: X\to\mathbb R$ be a real function and $L\in\mathbb R$. Then
$\lim_{x\to r} f(x)  =  L$ if and only if $^*f(r + dx) \approx L$ for all non-zero infinitesimals $dx\in\mathbb R$ such that $r+dx\in{^*X}$. In the latter case we have $\lim_{x\to r} f(x) = \st(^*f(r + dx))$. 
\end{theorem}
	
	The \emph{infinitesimal part} of the above characterization was the \emph{Leibniz definition of derivative}. Notice that the above characterization of the concept of \emph{limit} involves \emph{only one quantifier} in sharp contrast with the usual $\varepsilon, \delta$-definition of \emph{limit} in the modern real analysis using three non-commuting quantifiers. On the topic of \emph{counting the quantifiers} we refer to Cavalcante~\cite{cavalcante}.

\item Let $\rho$ be a positive infinitesimal in $^*\mathbb{R}$ (i.e. $0<\rho<1/n$ for all $n\in\mathbb N$). We define the sets of non-standard $\rho$-\emph{moderate} and $\rho$-\emph{negligible} numbers by 
  \begin{eqnarray*}
    \mathcal M_{\rho}(^*\mathbb R)&=&\{x\in{^*\mathbb R} : |x|\le \rho^{-m} \emph{ for some } m\in\mathbb N\},\\
    \mathcal N_{\rho}(^*\mathbb R)&=&\{x\in{^*\mathbb R} : |x| < \rho^n \emph{ for all } n\in\mathbb N\},
  \end{eqnarray*}
  respectively. The \emph{Robinson field of real $\rho$-asymptotic numbers} is the factor ring ${^{\rho}\mathbb R}=:\mathcal M_{\rho}(^*\mathbb R)/\mathcal N_{\rho}(^*\mathbb R)$. We denote by $\widehat{x}\in{^{\rho}\mathbb R}$ the equivalence class of $x\in\mathcal M_{\rho}(^*\mathbb R)$.  As it is not hard to show that $\mathcal M_{\rho}(^*\mathbb R)$ is a convex subring, and $\mathcal N_{\rho}(^*\mathbb R)$ is a maximal convex ideal; thus $^{\rho}\mathbb R                                             $ is an ordered field.  We observe that $^{\rho}\mathbb R$ is not algebraically saturated, since the sequence $\{\widehat{\rho}^{\;-n}\}_{n\in\mathbb N}$ is unbounded and increasing in $^\rho\mathbb R$. Consequently, ${\rm cof}(^\rho\mathbb R)=\aleph_0$ and $^\rho\mathbb R$ is Cauchy (and sequentially) complete. The field $^{\rho}\mathbb R$ was introduced by A. Robinson in (Robinson~\cite{aRob73}) and in (Lightstone \& Robinson~\cite{LiRob}). The proof that $^\rho\mathbb{R}$ is real-closed and Cantor complete can be found in (Todorov \& Vernaeve~\cite{TodVern08}, Theorem 7.3 and Theorem 10.2, respectively). The field $^{\rho}\mathbb R$ is also known as \emph{Robinson's valuation field}, because the mapping $v_\rho:{^{\rho}\mathbb R}\to\mathbb R\cup\{\infty\}$ defined by $v_\rho(\widehat{x})=\st(\log_\rho(|x|))$ if $\widehat{x}\not=0$, and $v_\rho(0)=\infty$, is a non-Archimedean valuation. $^{\rho}\mathbb R$ is also spherically complete (Luxemburg~\cite{wLux76}). We sometimes refer to the branch of mathematics related directly or indirectly to Robinson's field $^\rho\mathbb R$  as \emph{non-standard asymptotic analysis} (see the introduction in Todorov \& Vernaeve~\cite{TodVern08}).
%_______________________

\item By a result due to Robinson~\cite{aRob73} the filed $\mathbb R\langle t^{\mathbb R}\rangle$ can be embedded as an ordered subfield of ${^\rho\mathbb R}$, where the image of $t$ is $\widehat{\rho}$. We shall write this embedding as an inclusion, $\mathbb R\langle t^{\mathbb R}\rangle\subset{^\rho\mathbb R}$. The latter implies the chain of inclusions (embeddings):            
\begin{equation}\label{E: Chain 2}
\mathbb R\subset \mathbb R(t)\subset \mathbb R(t^{\mathbb Z})\subset \mathbb R\langle t^{\mathbb R}\rangle\subset{^\rho\mathbb R}.
\end{equation}
These embeddings explain the name \emph{asymptotic numbers} for the elements of $^\rho\mathbb R$. Recently it was shown that the fields $^*\mathbb R((t^{\mathbb R}))$ and ${^\rho\mathbb R}$ are ordered filed isomorphic  (Todorov \& Wolf~\cite{TodWolf04} ). Since $\mathbb R((t^{\mathbb R}))\subset{^*\mathbb R}((t^{\mathbb R}))$, the chain  (\ref{E: Chain 2}) implies two more chains:
\begin{align}
&\mathbb R\subset \mathbb R(t)\subset \mathbb R(t^{\mathbb Z})\subset \mathbb R\langle t^{\mathbb R}\rangle\subset {\mathbb R((t^{\mathbb R}))}\subset{^\rho\mathbb R}, \label{E: Chain 3}\\
\label{E: Chain 4}
&{^*\mathbb R}\subset {^*\mathbb R(t)}\subset{^*\mathbb R(t^{\mathbb Z})}\subset {^*\mathbb R}\langle t^{\mathbb R}\rangle\subset {^\rho\mathbb R}.
\end{align} 
\end{ex-enum}

%_____________________

\section{The Purge of Infinitesimals from Mathematics}\label{S: The Purge of Infinitesimals from Mathematics}

In this section we offer a short survey on the history of infinitesimal calculus written in a \emph{polemic-like style}. The purpose is to refresh the memory of the readers on one hand, and to prepare them for the next section on the other, where we shall claim \emph{the main point of our article}. For a more detailed exposition on the subject we refer to the recent article by Borovik \& Katz~\cite{BorovikKatz}, where the reader will find more references on the subject. 

\begin{itemize}
\item The Infinitesimal calculus was founded as a mathematical discipline by Leibniz and Newton, but the origin of \emph{infinitesimals} can be traced back to Cavalieri, Pascal, Fermat, L'Hopital and even to Archimedes. The development of calculus culminated in Euler's mathematical inventions. Perhaps Cauchy was the last -- among the great mathematicians -- who still taught calculus (in \'{E}cole) and did research in terms of infinitesimals. We shall refer to this period of analysis as the \emph{Leibniz-Euler Infinitesimal Calculus} for short. 

\item There has hardly ever been a more fruitful and exciting period in mathematics than during the time the Leibniz-Euler infinitesimal calculus was developed. New important results were pouring down from every area of science to which the new method of infinitesimals had been applied -- integration theory, ordinary and partial differential equations, geometry, harmonic analysis, special functions, mechanics, astronomy and physics in general. The mathematicians were highly respected in the science community for having ``in their pockets'' a new powerful method for analyzing everything ``which is changing.'' We might safely characterize the  \emph{Leibniz-Euler Infinitesimal Calculus} as the ``golden age of mathematics.'' We should note that all  of the mathematical achievements of infinitesimal calculus have survived up to modern times. Furthermore, the infinitesimal calculus has never encountered logical paradoxes -- such as Russell's paradox in set theory.

\item Regardless of the brilliant success and the lack of (detected) logical paradoxes, doubts about the philosophical and mathematical legitimacy of the foundation of infinitesimal calculus started from the very beginning. The main reason for worry was one of the principles (axioms) -- now called the Leibniz principle -- which claims that there exists a non-Archimedean totally ordered field with very special properties (a non-standard extension of an Archimedean field -- in modern terminology). This principle is not intuitively believable, especially if compared with the axioms of Euclidean geometry. After all, it is much easier to imagine ``points, lines and planes'' around us, rather than to believe that such things like an ``infinitesimal amount of wine'' or ``infinitesimal annual income'' might possibly have counterparts in the real world. The mathematicians of the $17^\textrm{th}$ and $18^\textrm{th}$ centuries  hardly had any experience with non-Archimedean fields -- even the simplest such field $\mathbb Q(x)$ was never seriously considered as a ``field extension'' (in modern terms) of $\mathbb Q$. 

\item  Looking back with the eyes of modern mathematicians, we can now see that the Leibniz-Euler calculus was actually quite rigorous -- at least much more rigorous than perceived by many modern mathematicians today and certainly by Weierstrass, Bolzano and Dedekind, who started the \emph{reformation of calculus} in the second part of the $19^\textrm{th}$ century. All axioms (or principles) of the infinitesimal calculus were correctly chosen and eventually survived quite well the test of the modern non-standard analysis invented by A. Robinson in the 1960's. What was missing at the beginning of the $19^\textrm{th}$ century to complete this theory was a proof of the consistency of its axioms; such a proof requires -- from modern point of view -- only two more things: Zorn's lemma (or equivalently, the axiom of choice) and a construction of a complete totally ordered field from the rationals. 

\item Weierstrass, Bolzano and Dedekind, along with many others, started the \emph{reformation of calculus} by expelling the infinitesimals and replacing them by the concept of the \emph{limit}. Of course, the newly created \emph{real analysis} also requires Zorn's lemma, or the equivalent axiom of choice, but the $19^\textrm{th}$ century mathematicians did not worry about such ``minor details,'' because most of them (with the possible exception of Zermelo) perhaps did not realize that \emph{real analysis} cannot possibly survive without the axiom of choice. The status of Zorn's lemma and the  axiom of choice were clarified a hundred years later by P. Cohen, K. G\"{o}del and others. Dedekind however (along with many others) constructed an example of a complete field, later called the \emph{field of Dedekind cuts}, and thus proved the consistency of the axioms of the real numbers. This was an important step ahead compared to the infinitesimal calculus.

	\item The \emph{purge of the infinitesimals} from calculus and from mathematics in general however came at a very high price (paid nowadays by the modern students in real analysis): the number of quantifiers in the definitions and theorems in the new \emph{real analysis} was increased by at least two additional quantifiers when compared to their counterparts in the \emph{infinitesimal calculus}. For example, the definition of a \emph{limit} or \emph{derivative} of a function in the Leibniz-Euler infinitesimal calculus requires only one quantifier (see Theorem~\ref{T: Leibniz Derivative}). In contrast, there are three non-commuting quantifiers in their counterparts in real analysis. In the middle of the $19^\textrm{th}$ century however, the word ``infinitesimals'' had become synonymous to ``non-rigorous'' and the mathematicians were ready to pay about any price to get rid of them. 

\item Starting from the beginning of the $20^\textrm{th}$ century \emph{infinitesimals} were systematically expelled from mathematics -- both from textbooks and research papers. The name of the whole discipline \emph{infinitesimal calculus} became to sound archaic and was first modified to \emph{differential calculus}, and later to simply \emph{calculus}, perhaps in an attempt to erase even the slightest remanence of the \emph{realm of infinitesimals}. Even in the historical remarks spread in the modern real analysis textbooks, the authors often indulge of a sort of rewriting the history by discussing the history of infinitesimal calculus, but not even mentioning the word ``infinitesimal.'' A contemporary student in mathematics might very well graduate from college without ever hearing anything about infinitesimals. 

\item The \emph{differentials} were also not spared from the purge -- because of their historical connection with infinitesimals. Eventually they were ``saved'' by the differential geometers under a new name: the \emph{total differentials} from infinitesimal calculus were renamed in  modern geometry to \emph{derivatives}. The sociologists of science might take note: it is not unusual in \emph{politics} or other more ideological fields to ``save'' a concept or idea by simply ``renaming it,'' but in mathematics this happens very, very rarely. The last standing reminders of the ``once proud infinitesimals'' in the modern mathematics are perhaps the ``symbols'' $dx$ and $dy$ in the Leibniz notation $dy/dx$ for derivative and in the integral $\int f(x)\, dx$, whose resilience  turned out to be without precedent in mathematics. An innocent and confused student in a modern calculus course, however, might ponder for hours over the question what the deep meaning (if any) of the word ``symbol" is.

\item In the 1960's,  A. Robinson invented non-standard analysis and Leibniz-Euler infinitesimal calculus was completely and totally rehabilitated. The attacks against the infinitesimals finally ceased, but the straightforward hatred toward them remains -- although rarely  expressed openly anymore. (We have reason to believe that the second most hated notion in mathematics after ``infinitesimals'' is perhaps ``asymptotic series,'' but this is a story for another time.) In the minds of many, however, there still remains the lingering suspicion that non-standard analysis is a sort of ``trickery of overly educated logicians'' who -- for lack of anything else to do -- ``only muddy the otherwise crystal-clear waters of modern real analysis.'' 

\item Summarizing the above historical remarks,  our overall impression is -- said figuratively -- that most modern mathematicians perhaps feel much more grateful to Weierstrass, Bolzano and Dedekind, than to Leibniz and Euler. And many of them perhaps would feel much happier now if the non-standard analysis had never been invented.
\end{itemize}
%_______________________
\section{How Rigorous Was the Leibniz-Euler Calculus}\label{S: How Rigorous Was the Leibniz-Euler Calculus}

 The Leibniz-Euler infinitesimal calculus was based on the existence of two totally ordered fields -- let us denote them by $\mathbb L$ and $^*\mathbb L$. We shall call $\mathbb L$ the \emph{Leibniz field} and $^*\mathbb L$ its \emph{Leibniz extension}. The identification of these fields has been a question of debate up to present times. What is known with certainty is the following: (a) $\mathbb L$ is an Archimedean field and $^*\mathbb L$ is a non-Archimedean field (in the modern terminology); (b)  ${^*\mathbb L}$ is a proper order field extension of $\mathbb L$; (c) $\mathbb L$ is Leibniz complete (see Axiom 1 below); (d) $\mathbb L$ and $^*\mathbb L$ satisfy the \emph{Leibniz Transfer Principle} (Axiom 2 below). 

\begin{remark}[About the Notation] The set-notation we just used to describe the infinitesimal calculus -- such as $\mathbb L, {^*\mathbb L}$, as well as $\mathbb N, \mathbb Q, \mathbb R$ ,etc. -- were never used in the $18^\textrm{th}$ century, nor for most of the $19^\textrm{th}$ century. Instead, the elements of $\mathbb L$ were described verbally as the ``usual quantities'' in contrast to the elements of $^*\mathbb L$  which were described in terms of infinitesimals: $dx, dy, 5+dx$, etc.. In spite of that, we shall continue to use the usual set-notation to facilitate the discussion. 
\end{remark}

	One of the purposes of this article is to try to convince the reader that the above assumptions for $\mathbb L$ and $^*\mathbb L$ imply that $\mathbb L$ is a complete field and thus isomorphic to the field of reals, $\mathbb R$. That means that the Leibniz-Euler infinitesimal calculus was already a rigorous branch of mathematics -- at least much more rigorous than many contemporary mathematicians prefer to believe. Our conclusion is that the amazing success of the infinitesimal calculus in science was possible, we argue, \textbf{not in spite of lack of rigor, but because of the high mathematical rigor} already embedded in the formalism. 

	All of this is in sharp contrast to the prevailing perception among many contemporary mathematicians that the Leibniz-Euler infinitesimal calculus was a non-rigorous branch of mathematics. Perhaps, this perception is due to the wrong impression which most modern calculus textbooks create. Here are several popular myths about the level of mathematical rigor of the infinitesimal calculus. 

\begin{myth} 
   Leibniz-Euler calculus was non-rigorous, because it was based  on the concept of \emph{non-zero infinitesimals}, rather than on \emph{limits}. The concept of \emph{non-zero infinitesimal} is perhaps, ``appealing for the intuition,'' but it is certainly mathematically non-rigorous. ``There is no such thing as a non-zero infinitesimal.'' The \emph{infinitesimals} should be expelled from mathematics once and for all, or perhaps, left only to the applied mathematicians and physicists to facilitate their intuition.
\end{myth}

\begin{fact}
   Yes, the Archimedean fields do not contain non-zero infinitesimals. In particular, $\mathbb R$ does not have non-zero infinitesimals. But in mathematics there are also non-Archimedean ordered fields and each such field contains infinitely many non-zero infinitesimals.  The simplest example of a non-Archimedean field is, perhaps, the field $\mathbb R(t)$ of rational functions with real coefficients supplied with ordering as in Example~\ref{Ex: Field of Rational Functions} in this article. Actually, every totally ordered field which contains a proper copy of $\mathbb R$ is non-Archimedean (see Section~\ref{S: Examples of Non-Archimedean Fields}). Blaming the \emph{non-zero infinitesimals} for the lack of rigor is \emph{nothing other than mathematical ignorance}! 
\end{fact}

\begin{myth}
   The Leibniz-Euler infinitesimal calculus was non-rigorous because of the lack of the completeness of the field of ``ordinary scalars'' $\mathbb L$.  Perhaps $\mathbb L$ should be identified with the field of rationals $\mathbb Q$, or the field $\mathbb A$ of the real algebraic numbers? Those who believe that the  Leibniz-Euler infinitesimal calculus was based on a non-complete field -- such as $\mathbb Q$ or $\mathbb A$ -- must face a very confusing mathematical and philosophical question: How, for God's sake, such a non-rigorous and naive framework as the field of rational numbers $\mathbb Q$ could support one of the most successful developments in the history of mathematics and science in general? Perhaps mathematical rigor is irrelevant to the success of mathematics? Or, even worst, mathematical rigor should be treated as an ``obstacle'' or ``barrier'' in the way of the success of science. This point of view is actually pretty common among applied mathematicians and theoretical physicists. We can only hope that those who teach real analysis courses nowadays do not advocate these values in class.
\end{myth}

\begin{fact}
   The answer to the question ``was the Leibniz field $\mathbb L$ complete'' depends on whether or not the Leibniz extension $^*\mathbb L$ can be viewed as a ``non-standard extension'' of $\mathbb L$ in the sense of the modern non-standard analysis. Why? Because our result in Theorem~\ref{T: Completeness of an Archimedean Field} of this article shows that if the Leibniz extension $^*\mathbb L$ of $\mathbb L$ is in fact a non-standard extension of $\mathbb L$, then $\mathbb L$ is a  complete Archimedean field which is thus isomorphic to the field of real numbers $\mathbb R$. On the other hand, there is plenty of evidence that Leibniz and Euler, along with many other mathematicians, had regularly employed the following principle:

   \begin{axiom}[Leibniz Completeness Principle]
      Every finite number in $^*\mathbb L$ is infinitely close to some (necessarily unique) number in $\mathbb L$ (\#2 of Definition~\ref{D: Completeness of a Totally Ordered Field}).
   \end{axiom}
\end{fact}

\begin{remark} The above property of $^*\mathbb L$ was treated by Leibniz and the others as an ``obvious truth.'' More likely, the $18^\textrm{th}$ century mathematicians were unable to imagine a counter-example to the above statement. The results of non-standard analysis produce such a counter-example: there exist finite numbers in $^*\mathbb Q$ which are \emph{not} infinitely close to any number in $\mathbb Q$.
\end{remark}

\begin{myth}
   The theory of non-standard analysis is an invention of the $20^\textrm{th}$ century and has nothing to do with the Leibniz-Euler infinitesimal calculus. We should not try to rewrite the history and project backwards the achievements of modern mathematics. The proponents of this point of view also emphasize the following features of non-standard analysis:
   \begin{enumerate}
   \item[(a)] A. Robinson's original version of non-standard analysis was based of the so-called \emph{Compactness Theorem} from model theory: If a set $S$ of sentences has the property that every finite subset of $S$ is consistent (has a model), then the whole set $S$ is consistent (has a model). 

   \item[(b)] The ultrapower construction of the non-standard extension $^*\mathbb K$ of a field $\mathbb K$ used in Example 5, Section \ref{S: Examples of Non-Archimedean Fields}, is based on the existence of a free ultrafilter. Nowadays we prove the existence of such a filter with the help of Zorn's lemma.  Actually the statement that \emph{for every infinite set $I$ there exists a free ultrafilter on $I$} is known in modern set theory  as the \emph{free filter axiom} (an axiom which is weaker than the axiom of choice).
   \end{enumerate} 
\end{myth}

\begin{fact}
   We completely and totally agree with both (a) and (b) above. Neither the \emph{completeness theorem} from model theory, nor the \emph{free filter axiom} can be recognized  in any shape or form in the Leibniz-Euler infinitesimal calculus.  These inventions belong to the late $19^\textrm{th}$ and the first half of $20^\textrm{th}$ century. Perhaps surprisingly for many of us, however, J. Keisler~\cite{jKeislerF} invented a simplified version of non-standard analysis -- general enough to support calculus -- which does not rely on either model theory or formal mathematical logic. It presents the definition of $^*\mathbb R$ axiomatically in terms of a particular extension of all functions from $\mathbb L$ to $^*\mathbb L$ satisfying the so-called \emph{Leibniz Transfer Principle}:

   \begin{axiom}[Leibniz Transfer Principle]\label{A: Leibniz Transfer Principle} 
      For every $d\in\mathbb N$ and for every set $S\subset\mathbb L^d$ there exists a unique set $^*S\subset{^*\mathbb L^d}$ such that:
      \begin{enumerate}
      \item[{\bf (a)}] $^*S \cap \mathbb L^d=S$.
      \item[{\bf (b)}] If $f\subset \mathbb L^p\times \mathbb L^q$ is a function, then $^*f\subset{^*\mathbb L^p}\times{^*\mathbb L^q}$ is also a function. 
      \item[{\bf (c)}] $\mathbb L$ satisfies Theorem~\ref{T: Leibniz Transfer Principle} for $\mathbb K=\mathbb L$. 
      \end{enumerate}
   \end{axiom}

   We shall call $^*S$ a \emph{non-standard extension} of $S$ borrowing the terminology from Example 5, Section~\ref{S: Examples of Non-Archimedean Fields}. 
\end{fact}

\begin{examples} Here are two (typical) examples for the application of the Leibniz transfer principle:
   \begin{ex-enum}
   \item The identity $\sin(x+y)={\sin x}\, {\cos y}+{\cos x}\, {\sin y}$ holds for all $x, y\in{\mathbb L}$ if and only if this  identity holds for all $x, y\in{^*\mathbb L}$ (where the asterisks in front of the sine and cosine are skipped for simplicity). 
   \item The identity $\ln(xy)=\ln x +\ln y$ holds for all $x, y\in{\mathbb L_+}$ if and only if the same  identity holds for all $x, y\in{^*\mathbb L_+}$. Here $\mathbb L_+$ is the set of the positive elements of $\mathbb L$ and $^*\mathbb L_+$ is its non-standard extension  (where again, the asterisks in front of the functions are skipped for simplicity). 
   \end{ex-enum}
\end{examples} 

 Leibniz never formulated his principle exactly in the form presented above. For one thing, the set-notation such as $\mathbb N$, $\mathbb Q$, $\mathbb R$, $\mathbb L$, etc. were not in use in the $18^\textrm{th}$ century. The name ``Leibniz Principle'' is often used in the modern literature (see Keisler~\cite{jKeislerF}, pp. 42 or Stroyan \& Luxemburg~\cite{StroyanLux76}, pp. 28), because Leibniz suggested that the field of the usual numbers ($\mathbb L$ or $\mathbb R$ here) should be extended to a larger system of numbers ($^*\mathbb L$ or $^*\mathbb R$), which has the same properties, but contains infinitesimals. Both Axiom 1 and Axiom 2 however, were in active use in Leibniz-Euler infinitesimal calculus.  Their implementation does not require an ellaborate set theory or formal logic; \emph{what is a solution of a system of equations and inequalities} was perfectly clear to mathematicians long before of the times of Leibniz and Euler. Both Axiom 1 and Axiom 2 are theorems in modern non-standard analysis (Keisler~\cite{jKeislerF}, pp. 42). However, if Axiom 1 and Axiom 2 are understood as \emph{proper axioms}, they characterize the field $\mathbb L$ uniquely (up to a field isomorphism) as a complete Archimedean field (thus a field isomorphic to $\mathbb R$). Also, these axioms characterize $^*\mathbb L$ as a \emph{non-standard extension} of $\mathbb L$. True, these two axioms do not determine $^*\mathbb L$ uniquely (up to a field isomorphism) unless we borrow from the modern set theory such tools as \emph{cardinality}. For the rigorousness of the infinitesimal this does not matter. Here is an example how the formula $(\sin x)^\prime=\cos x$ was derived in the infinitesimal calculus: suppose that $x\in \mathbb L$. Then for every non-zero infinitesimal $dx\in{^*\mathbb L}$ we have 
\begin{align}
&\frac{\sin(x+dx)-\sin x}{dx}= \frac{{\sin x}\, {\cos (dx)}+{\cos x}\, {\sin (dx)}-\sin x}{dx}=\notag\\
&=\sin x\, \frac{\cos(dx)-1}{dx}+\cos x\, \frac{\sin(dx)}{dx}\approx \cos x,\notag
\end{align}
because $\sin x, \cos x\in\mathbb L$, $\frac{\cos(dx)-1}{dx}\approx 0$ and $\frac{\sin(dx)}{dx}\approx 1$. Here $\approx$ stands for the infinitesimal relation on $^*\mathbb L$, i.e. $x\approx y$ if $x-y$ is infinitesimal (Definition~\ref{D: Infinitesimals, etc.}). Thus $(\sin x)^\prime=\cos x$ by the Leibniz definition of derivative (Theorem~\ref{T: Leibniz Derivative}). For those who are interested in teaching calculus through infinitesimals we refer to the calculus textbook Keisler~\cite{jKeislerE} and its more advanced companion Keisler~\cite{jKeislerF} (both available on internet). On a method of teaching limits through infinitesimals we refer to Todorov~\cite{tdTod2000a}.

\begin{myth}
   Trigonometry and the theory of algebraic and transcendental elementary functions such as $\sin x, e^x$, etc. was not rigorous in the infinitesimal calculus. After all, the theory of analytic functions was not developed until late in $19^\textrm{th}$ century.   
\end{myth}

\begin{fact}
   Again, we argue that the rigor of trigonometry and the elementary functions was relatively high and certainly much higher than in most of the contemporary trigonometry and calculus textbooks. In particular, $y=\sin x$ was defined by first, defining $\sin^{-1} y=\int_0^x\frac{dy}{\sqrt{1-y^2}}$ on $[-1, 1]$ and geometrically viewed as a particular \emph{arc-lenght on a circle}. Then $\sin x$ on $[-\pi/2, \pi/2]$ is defined as the inverse of $\sin^{-1} y$. If needed, the result can be extended to $\mathbb L$ (or to $\mathbb R$) by periodicity. This answer leads to another question: how was the \emph{concept of arc-lenght} and the integral $\int_a^b f(x)\, dx$ defined in terms of infinitesimals before the advent of Riemann's theory of integration? On this topic we refer the curious reader to Cavalcante \& Todorov~\cite{CavTod08}.
\end{fact}

\textbf{James Hall} is an undergraduate student  in mathematics at Cal Poly, San Luis Obispo. In the coming Fall of 2011 he will be a senior and will start to apply for graduate study in mathematics. His recent senior project was on \emph{Completeness of Archimedean and non-Archimedean fields and Non-Standard Analysis}.

Mathematics Department, California Polytechnic State University, San Luis Obispo, CA 93407, USA  (james@slohall.com).

\textbf{Todor D. Todorov}  received his Ph.D. in Mathematical Physics from University of Sofia, Bulgaria. He currently teaches mathematics at Cal Poly, San Luis Obispo. His articles are on non-standard analysis, nonlinear theory of generalized functions (J. F. Colombeau's theory), asymptotic analysis, compactifications of ordered topological spaces, and teaching calculus. Presently he works on an axiomatic approach to the nonlinear theory of generalized functions. 

Mathematics Department, California Polytechnic State University, San Luis Obispo, CA 93407, USA  (ttodorov@calpoly.edu).

%____________________

%_____________________

\end{document}